\documentclass[12pt]{amsart}

\usepackage[colorlinks=true, pdfstartview=FitV, linkcolor=blue,
citecolor=blue, urlcolor=blue]{hyperref}

\usepackage{mathtools}
\usepackage{amsfonts}
\usepackage{amsmath}
\usepackage{amssymb} 
\usepackage{amsthm}

\usepackage{times}
\usepackage[T1]{fontenc}
\usepackage{enumitem}
\usepackage{setspace}
\usepackage{microtype}
\usepackage{cite}

\allowdisplaybreaks

\usepackage[margin=1.2in]{geometry}

\newtheorem{theorem}{Theorem}[section]
\newtheorem{lemma}[theorem]{Lemma}
\newtheorem{prop}[theorem]{Proposition}

\theoremstyle{definition}

\theoremstyle{remark}
\newtheorem{remark}[theorem]{Remark}

\numberwithin{equation}{section}

\newcommand{\R}{\mathbb R}

\newcommand{\rn}{{{\mathbb R}^n}}

\newcommand{\RR}{\mathbb{R}}

\DeclareMathOperator{\supp}{supp}

\DeclareMathOperator*{\essinf}{ess\,inf}
\DeclareMathOperator*{\esssup}{ess\,sup}

\DeclareRobustCommand{\rchi}{{\mathpalette\irchi\relax}}
\newcommand{\irchi}[2]{\raisebox{\depth}{$#1\chi$}}
\newcommand{\avg}[1]{\langle #1 \rangle}

\newcommand{\D}{\mathcal D}
\newcommand{\Ss}{\mathcal S}

\newcommand{\op}[1]{| #1 |_{\mathrm{op}}}
\newcommand{\A}{\mathcal A}
\newcommand{\W}{\mathcal W}
\newcommand{\V}{\mathcal V}

\newcommand{\Asc}{\mathrm{sc}}

\def\Xint#1{\mathchoice
   {\XXint\displaystyle\textstyle{#1}}%
   {\XXint\textstyle\scriptstyle{#1}}%
   {\XXint\scriptstyle\scriptscriptstyle{#1}}%
   {\XXint\scriptscriptstyle\scriptscriptstyle{#1}}%
   \!\int}
\def\XXint#1#2#3{{\setbox0=\hbox{$#1{#2#3}{\int}$}
     \vcenter{\hbox{$#2#3$}}\kern-.5\wd0}}

\def\avgint{\Xint-}

\begin{document}

\title[Weighted weak-type inequalities]
{Weighted weak-type inequalities for maximal operators and singular integrals}

\author{David Cruz-Uribe, OFS}
\address{Dept. of Mathematics \\
University of Alabama \\
 Tuscaloosa, AL 35487, USA}
\email{dcruzuribe@ua.edu}

\author{Brandon Sweeting}
\address{Dept. of Mathematics \\
University of Alabama \\
 Tuscaloosa, AL 35487, USA}
\email{bssweeting@ua.edu}


\date{November 1, 2023}

\thanks{The first author is partially supported by a Simons Foundation
  Travel Support for Mathematicians Grant.}

\begin{abstract}
  We prove quantitative, one-weight, weak-type estimates for maximal
  operators, singular integrals, fractional maximal operators and
  fractional integral operators.  We consider a kind of weak-type
  inequality that was first studied by Muckenhoupt and
  Wheeden~\cite{MR447956} and later in~\cite{MR2172941}.  We
  obtain quantitative estimates for these operators in both the scalar
  and matrix weighted setting  using
  sparse domination techniques.  Our results extend those obtained by Cruz-Uribe,
  Isralowitz, Moen, Pott, and Rivera-Ríos~\cite{MR4269407} for singular
  integrals and maximal operators when $p=1$.
\end{abstract}

\maketitle

\section{Introduction}

In this paper we consider weighted weak-type inequalities in a form
first introduced by Muckenhoupt and Wheeden.  To state our results we
first review some history; for conciseness we defer the statement of
relevant definitions to Section~\ref{section:prelim}. 

Typically, for $1\leq p<\infty$, given a weight $w$, a weak $(p,p)$
inequality for an operator $T$ refers to an inequality of the form
\begin{equation} \label{eqn:weak-measure}
w(\{x \in \rn : |Tf(x)|> \lambda\})
  \leq \frac{C}{\lambda^p} \int_\rn |f|^pw\,dx. 
\end{equation}
Such inequalities follow from the corresponding strong $(p,p)$
inequalities
\[ \int_\rn |Tf|^pw\,dx \leq C\int_\rn |f|^pw\,dx \]
by Chebyshev's inequality, treating $w$ as a measure.  However, if we
treat $w$ as a multiplier, and rewrite the strong-type inequality by
replacing $f$ by $fw^{-\frac{1}{p}}$ we get a strong-type inequality of the
form
\[ \int_\rn |w^{\frac{1}{p}}T(fw^{-\frac{1}{p}})|^p\,dx \leq C\int_\rn |f|^p\,dx. \]
Again by Chebyshev's inequality, this implies a weak-type inequality
of the form
\begin{equation} \label{eqn:weak-mult}
 |\{x \in \rn : |w(x)^{\frac{1}{p}}T(fw^{-\frac{1}{p}})(x)|> \lambda\}|
      \leq \frac{C}{\lambda^p} \int_\rn |f|^p\,dx. 
    \end{equation}
    These inequalities were first considered by Muckenhoupt and Wheeden
    in~\cite{MR447956}: when $n=1$ they proved that they hold for
    $1\leq p<\infty$ for the Hardy-Littlewood maximal operator and the
    Hilbert transform provided $w$ is in the Muckenhoupt class $A_p$.
    Their work was extended to higher dimensions and all singular
    integrals  by the first author, Martell and P\'erez~\cite{MR2172941}.  To
    distinguish this kind of inequality from~\eqref{eqn:weak-measure},
    we will refer to inequalities like~\eqref{eqn:weak-mult} as {\em
      multiplier weak-type
    inequalities}.

  \begin{remark}
    Two-weight generalizations of multiplier weak-type inequalities were introduced by
Sawyer~\cite{MR776188} when $n=1$.  His results were extended to higher
dimensions in~\cite{MR2172941},
and have since been considered by a number of
authors:  see, for
example,~\cite{MR3961329,MR3850676,MR3498179,MR4002540}.
\end{remark}

Recently, there has been a renewed interest in multiplier weak-type
inequalities.  In~\cite{MR4269407}, the first author, Isralowitz,
Moen, Pott, and Rivera-R\'\i os  showed that they are
the correct approach to generalize weak-type inequalities to matrix
$A_p$ weights. They proved the following quantitative estimates for singular
integrals and the Christ-Goldberg maximal operator.

\begin{theorem} \label{thm:p=1-matrix}
  Given a matrix weight $W\in \mathcal{A}_1$ and a Calder\'on-Zygmund singular integral
  operator $T$, then for every $f\in L^1(\rn,\RR^d)$ and $\lambda>0$,
  \[ |\{ x\in \rn : |W(x)T(W^{-1}f)(x)|> \lambda \}|
    \leq C(n,p,T)[W]_{\mathcal{A}_1}[W]_{A_\infty^\Asc}\frac{1}{\lambda}
    \int_\rn|f|\,dx. \]
  Similarly, for the Christ-Goldberg maximal operator $M_W$ we have
  \[ |\{ x\in \rn : M_Wf(x)> \lambda \}|
    \leq C(n,p)[W]_{\mathcal{A}_1}[W]_{A_\infty^\Asc}\frac{1}{\lambda}
    \int_\rn|f|\,dx. \]
\end{theorem}

As a corollary to Theorem~\ref{thm:p=1-matrix}
they proved a quantitative version of the results of Muckenhoupt and
Wheeden when $p=1$.

\begin{theorem} \label{thm:p=1-scalar}
  Given $w\in A_1$ and a Calder\'on-Zygmund singular integral operator
  $T$, then for every $f\in L^1$ and $\lambda>0$,
\[   |\{x \in \rn : |w(x)T(fw^{-1})(x)|> \lambda\}|
      \leq C(n,T)[w]_{A_1}[w]_{A_\infty}\frac{1}{\lambda} \int_\rn
      |f|\,dx.  \]
 The same inequality holds if $T$ is replaced by the Hardy-Littlewood
 maximal operator $M$.
\end{theorem}

\begin{remark}
It is an immediate that in Theorem~\ref{thm:p=1-scalar} we can replace
  $[w]_{A_1}[w]_{A_\infty}$ by the larger constant $[w]_{A_1}^2$, and
  the analogous observation holds for Theorem~\ref{thm:p=1-matrix}.
\end{remark}

\medskip

The main results of this paper are two-fold.  First, we extend
Theorems~\ref{thm:p=1-matrix} and~\ref{thm:p=1-scalar} to $p>1$ and $A_p$
weights.  The scalar theorem is just a special case of the
matrix theorem, but we state and prove it separately, since the proof
of the matrix theorem depends on the scalar proof.

\begin{theorem}\label{thm:sio-scalar}
 Given $1<p <\infty$,  $w \in A_p$, and a Calder\'on-Zygmund singular integral operator
  $T$, then for every $f \in L^p$ and $\lambda > 0$,
  \begin{equation} \label{eqn:sio-scalar1}
|\{x \in \mathbb{R}^d : |w(x)^{\frac{1}{p}}T(fw^{-\frac{1}{p}})(x)| > \lambda
        \}| \leq
        C(n,p,T) [w]_{A_p}[w]_{A_\infty}^p
        \frac{1}{\lambda^p}\int_\rn|f|^p\,dx. 
      \end{equation}
 The same inequality holds if $T$ is replaced by the Hardy-Littlewood
 maximal operator $M$.
\end{theorem}

\begin{theorem}\label{thm:sio-matrix}
 Given $1< p <\infty$,  a matrix weight $W \in \mathcal{A}_p$, and a Calder\'on-Zygmund singular integral operator
  $T$, then for every $f \in L^p(\rn,\RR^d)$ and $\lambda > 0$,
  \begin{equation} \label{eqn:sio-matrix}
|\{x \in \mathbb{R}^d : |W(x)^{\frac{1}{p}}T(W^{-\frac{1}{p}}f)(x)| > \lambda
        \}| \leq
        C(n,p,T) [W]_{\mathcal{A}_p}[W]_{A_\infty^\Asc}^p
        \frac{1}{\lambda^p}\int_\rn|f|^p\,dx. 
      \end{equation}
      Similarly, for the Christ-Goldberg maximal operator we have
      \begin{equation} \label{eqn:sio-matrix}
|\{x \in \mathbb{R}^d : M_Wf(x) > \lambda
        \}| \leq
        C(n,p) [W]_{\mathcal{A}_p}[W]_{A_\infty^\Asc}^p
        \frac{1}{\lambda^p}\int_\rn|f|^p\,dx. 
      \end{equation}
\end{theorem}

As noted above, multiplier weak-type inequalities follow from strong-type inequalities, so in the scalar case we can  obtain quantitative
estimates for singular integrals from the so-called ``$A_2$
theorem'' of Hyt\"onen~\cite{MR2912709} (see also~\cite{MR3000426}),
or the mixed-type inequalities of Hyt\"onen and P\'erez~\cite{MR3092729} and Hyt\"onen
and Lacey~\cite{MR3129101}; similarly, we can obtain results for the Hardy-Littlewood
maximal operator from the sharp strong-type estimates (see~\cite{dcu-paseky}).   For simplicity, we will only consider the weaker
estimates involving the $A_p$ characteristic.
Theorem~\ref{thm:sio-scalar} gives the estimate
\[ \| w^{\frac{1}{p}}T(\cdot\,w^{-\frac{1}{p}}) \|_{L^p \rightarrow
    L^{p,\infty}} \lesssim [w]_{A_p}^{\frac{1}{p}+1}.  \]
  From
  the $A_2$ theorem for $p>1$ we get the estimate
\[ \| w^{\frac{1}{p}}T(\cdot\,w^{-\frac{1}{p}}) \|_{L^p \rightarrow
    L^{p,\infty}} \lesssim [w]_{A_p}^{\max(1,\frac{1}{p-1})}. \]  Thus, we get a sharper estimate in the
range $1<p<\frac{1+\sqrt{5}}{2}$.   Similarly, for maximal operators,
the sharp constant in the strong $(p,p)$ inequality
(see~\cite{dcu-paseky}) gives a constant $[w]_{A_p}^\frac{1}{p-1}$, so again
Theorem~\ref{thm:sio-scalar}  gives a sharper constant when $p$ is
in this range.  In the matrix case, the same sharp bounds hold for the
Christ-Goldberg maximal operator (see~\cite{MR4030471}), so we can make
similar estimates.  For singular integral operators, it is conjectured
that the $A_2$ conjecture holds in the matrix case
(see~\cite{MR3689742}; currently, however, the best quantitative
estimate is $[W]_{\mathcal{A}_p}^{1+\frac{1}{p-1}-\frac{1}{p}}$, so our
estimates are sharper than those gotten from the best known
strong-type estimates when $1<p<2$.   These quantitative results for
matrix weights answer a question first raised
in~\cite[Remark~1.6]{MR4269407}.  

\medskip

We believe that Theorem~\ref{thm:p=1-scalar} and
Theorem~\ref{thm:sio-scalar} are sharp in the range close to $1$ where
we get better estimates than from the strong-type inequality.  As
evidence for this conjecture, we note that very recently, Lerner, Li,
Ombrosi, and Rivera-R\'\i os~\cite{LLOR23} proved that the sharp
constant when $p=1$ in Theorem~\ref{thm:p=1-matrix} is $[W]_{\mathcal{A}_1}^2$
for both singular integrals and the Christ-Goldberg maximal operator,
and the corresponding bound is also sharp in the scalar case.
Independent of their work we proved a lower bound for the constant
that is worse their sharp estimate.  Though superceded, we think that
our example, which is simpler than theirs, gives some insight into the behavior of weights at
this endpoint.  Therefore, we sketch the details of the construction
in an appendix.

\begin{remark}
It is interesting to compare the bounds in Theorem~\ref{thm:sio-scalar} to
the sharp constants in the scalar weak-type inequalities of the
form~\eqref{eqn:weak-measure} for maximal
operators and singular integrals.  For the Hardy-Littlewood maximal operator, the sharp
constant in inequality~\eqref{eqn:weak-measure} is proportional to
$[w]_{A_p}^\frac{1}{p}$ for $p\geq 1$.  For singular integrals, the sharp constant
is proportional to $[w]_{A_1}\log(e+[w]_{A_1})$ when
$p=1$~\cite{MR4150264}, and $[w]_{A_p}$ when
$p>1$~\cite{MR2993026}.  However, these inequalities are not
completely comparable to those in Theorem~\ref{thm:sio-scalar}.
In~\cite{MR447956}, it was shown that~\eqref{eqn:sio-scalar1} holds for the
maximal operator for all $p\geq 1$ with $w(x)=|x|^{-1}$ when $n=1$; this $w$ is
not in $A_p$ for any $p\geq 1$.  Similarly,~\eqref{eqn:sio-scalar1}  holds
for the Hilbert transform when $p=1$ with the same weight.
Additional results about the necessary and sufficient conditions for
multiplier weak-type inequalities will appear elsewhere~\cite{brandon-23}.
\end{remark}

\medskip

In our second set of results we prove analogous theorems for the
fractional maximal operator and fractional integral operators in both
the scalar and matrix case.  We again state the results separately.

\begin{theorem} \label{thm:fractional-scalar}
  Given $0<\alpha<n$, fix $1\leq p<\frac{n}{\alpha}$ and define $q$ by
  $\frac{1}{p}-\frac{1}{q}=\frac{\alpha}{n}$.   Let $w\in A_{p,q}$.
  Then for every $f \in L^p$ and $\lambda > 0$,
  \begin{multline} \label{eqn:fractional-scalar1}
|\{x \in \mathbb{R}^d : |w(x)I_\alpha(w^{-1}f)(x)| > \lambda
\}| \\
\leq
        C(n,\alpha,p) [w]_{A_{p,q}}[w^q]_{A_\infty}^q
        \frac{1}{\lambda^q}\bigg(\int_\rn|f|^p\,dx\bigg)^{\frac{q}{p}}.
      \end{multline}
 The same inequality holds if we replace $I_\alpha$ by the fractional
 maximal operator $M_\alpha$. 
\end{theorem}

\begin{theorem} \label{thm:fractional-matrix}
  Given $0<\alpha<n$, fix $1\leq p<\frac{n}{\alpha}$ and define $q$ by
  $\frac{1}{p}-\frac{1}{q}=\frac{\alpha}{n}$.   Let $W\in \mathcal{A}_{p,q}$.
  Then for every $f \in L^p(\rn,\RR^d)$ and $\lambda > 0$,
  \begin{multline} \label{eqn:fractional-matrix}
|\{x \in \mathbb{R}^d : |W(x)I_\alpha(W^{-1}f)(x)| > \lambda
\}| \\
\leq
        C(n,\alpha,p) [W]_{\mathcal{A}_{p,q}}[W^q]_{A_\infty^\Asc}^q
        \frac{1}{\lambda^q}\bigg(\int_\rn|f|^p\,dx\bigg)^{\frac{q}{p}}.
      \end{multline}
 The same inequality holds if we replace $WI_\alpha(W^{-1}\cdot)$  by the fractional
 Christ-Goldberg maximal operator $M_{W,\alpha}$.
\end{theorem}

We can again use quantitative strong-type inequalities to deduce
multiplier weak-type inequalities for fractional maximal operators and
fractional integrals.  Lacey, {\em et al.}~\cite{MR2652182} proved the
following sharp bounds:
\[ \|wM_\alpha f\|_{L^q}
  \leq C[w]_{A_{p,q}}^{(1-\frac{\alpha}{n})\frac{p'}{q}}
  \|wf\|_{L^p},
  \quad
\|wI_\alpha f\|_{L^q}
  \leq C[w]_{A_{p,q}}^{(1-\frac{\alpha}{n})\max(1,\frac{p'}{q})}
  \|wf\|_{L^p}.
\]
Moen and Isralowitz~\cite{MR4030471} showed that the Christ-Goldberg
fractional maximal operator $M_{W,\alpha}$ satisfied the same sharp
inequality; for the fractional integral operator they were only able
to prove a quantitative estimate with the larger exponent
$(1-\frac{\alpha}{n})\frac{p'}{q}+1$.  It is conjectured that the
sharp exponent for $I_\alpha$ in the matrix case is the same as in the
scalar case.  Thus, for the fractional
maximal operators in both the scalar and matrix case, we get a better estimate close to $1$:  more
precisely in the range
\[ 1 < p < \frac{\big(1-\frac{2\alpha}{n}\big)+
    \sqrt{5-\frac{4\alpha}{n}}}{2\big(1-\frac{\alpha^2}{n^2}\big)}. \]

\begin{remark}
  It is an open problem, even in the scalar case, to determine what the necessary and
  sufficient conditions are on $w$ for the multiplier weak-type
  inequality for
  the fractional maximal and integral operators.
\end{remark}

The remainder of this paper is organized as follows.  In
 Section~\ref{section:prelim} we gather some preliminary results on
 scalar and matrix weights, and define the operators we are interested
 in.  In
 Section~\ref{section:scalar} we prove
 Theorem~\ref{thm:sio-scalar} and~\ref{thm:fractional-scalar}, and in
 Section~\ref{section:matrix} we prove Theorems~\ref{thm:sio-matrix}
 and~\ref{thm:fractional-matrix}.   Finally, in
Appendix~\ref{section:example} we give a lower bound for the $p=1$ case
 by constructing an explicit example.

\section{Preliminaries}
\label{section:prelim}

Throughout this paper, $n$ will denote the dimension of the domain
$\rn$ of all functions; $d$ will denote the dimension of vector and
matrix valued functions.   By $C$, $c$, etc. we will mean constants
that depend only on underlying parameters but may otherwise change
from line to line.  If we write $A\lesssim B$, we mean that there
exists $c>0$ such that $A\leq cB$.  By a cube $Q$ we mean a cube in $\rn$
whose sides are parallel to the coordinate axes.

\subsection*{Scalar operators}
We will be working with  the following operators.  For more
information, see~\cite{duoandikoetxea01}.   Given $f\in
L^1_{loc}$, we define the Hardy-Littlewood maximal operator by
\[ Mf(x) := \sup_Q \avgint_Q |f(y)|\,dy \cdot \rchi_Q(x), \]
where the supremum is taken over all cubes $Q$. We will need the
following estimate, which follows from the standard proof of the
boundedness of $M$ via Marcinkiewicz interpolation:  for $1<p<\infty$,
$\|Mf\|_{L^p} \leq Cp'\|f\|_{L^p}$.

A Calder\'on-Zygmund singular integral is an operator $T$ that
is bounded on $L^2$, such that there exists a kernel function $K$
defined on $\rn \times \rn \setminus \{ (x,x) : x \in \rn\}$ so that
for all $f\in L^2_c$, and $x\not\in \supp(f)$,
\[ Tf(x) = \int_\rn K(x,y) f(y)\,dy.  \]
We assume that the kernel satisfies the size estimate
\[ |K(x,y)| \leq \frac{C}{|x-y|^n}, \]
and the regularity estimate
\[ |K(x+h,y)-K(x,y)|+|K(x,y+h)-K(x,y)| \leq
  C\frac{|h|^\delta}{|x-y|^{n+\delta}} \]
  whenever $|x-y|>2|h|$.  

  Given $0<\alpha<n$ and $f\in
L^1_{loc}$,  we define the fractional maximal operator by
\[ M_\alpha f(x) := \sup_Q |Q|^{\frac{\alpha}{n}} \avgint_Q |f(y)|\,dy \cdot \rchi_Q(x), \]
where the supremum is again taken over all cubes $Q$.  We define the fractional
integral operator $I_\alpha$ to be the convolution operator
\[ I_\alpha f(x) := \int_\rn \frac{f(y)}{|x-y|^{n-\alpha}}. \]

\medskip

\subsection*{Scalar weights} We will need the following facts about
scalar weights. For more information,
see~\cite{duoandikoetxea01,dcu-paseky,CruzUribe:2016ji}. Given
$1 < p < \infty$, we say that $w \in A_p$ if
\[
  [w]_{A_p} := \sup_Q \left(\avgint_Q{w(x)\,dx}\right)
  \left(\avgint_Q{w(x)^{1-p'}\,dx}\right)^{p-1} < \infty, 
\]
where the supremum is taken over cubes $Q$. For $p = 1$, we say $w \in A_1$ if 
\[
    [w]_{A_1} := \sup_Q \esssup_{x\in Q}
    w(x)^{-1}\left(\avgint_Q{w(x)\,dx}\right) < \infty , 
\]
where the supremum is again taken over cubes $Q$.  Define
$A_\infty= \bigcup_{p\geq 1}A_p$.  
A weight $w\in A_\infty$ if and
only if $w$ satisfies the reverse H\"older inequality for some exponent
$s>1$, denoted by $w\in RH_s$:
\[ [w]_{RH_s} := \sup_Q \bigg(\avgint_Q w(x)^s\,dx\bigg)^{\frac{1}{s}}
  \bigg(\avgint_Q w(x)\,dx\bigg)^{-1} < \infty, \]
where again the supremum is taken over all cubes.   We will need the
following sharp version of the reverse H\"older inequality due to
Hyt\"onen and P\'erez~\cite{MR3092729}.  Here, by $[w]_{A_\infty}$ we
mean the Fujii-Wilson $A_\infty$ characteristic of $w$.  The precise
definition does not matter for us, and we refer the reader
to~\cite{MR3092729} for more details.  

\begin{prop} \label{prop:sharp-rhi}
  Given $w\in A_\infty$, there exists a constant $c=c(n)$ such that if
  $\nu=1+c[w]_{A_\infty}^{-1}$, then $w\in RH_\nu$ and $[w]_{RH_\nu}\leq 2$.
\end{prop}

Finally, we define the fractional weight classes $A_{p,q}$.  Given
$0<\alpha<n$ and $1\leq p< \frac{n}{\alpha}$, define $q$ by
$\frac{1}{p}-\frac{1}{q}=\frac{\alpha}{n}$.  A weight $w$ is in $A_{p,q}$
if
\[
  [w]_{A_{p,q}} := \sup_Q \left(\avgint_Q{w(x)^q\,dx}\right)
  \left(\avgint_Q{w(x)^{-p'}\,dx}\right)^{\frac{q}{p'}} < \infty, 
\]
where the supremum is taken over all cubes $Q$.  A weight $w\in
A_{1,q}$ if
\[ [w]_{A_{1,q}} := \sup_Q \esssup_{x\in Q}
  w(x)^{-q} \bigg(\avgint_Q w(x)^q\,dx\bigg) <\infty. \]
Note that for all $p$ and $q$, $w\in A_{p,q}$ if and only if $w^q\in
A_{1+\frac{q}{p'}}$.  Consequently, $w^q\in A_\infty$
and so satisfies
a reverse H\"older inequality.

\medskip

\subsection*{Matrix weights and operators}  We need the following basic information
about matrix weights.  For more
details, see~\cite{MR4269407,MR3803292,MR3544941}.
Given a $d\times d$ matrix $W$, the operator norm of $W$ is defined to
be
\[ \op{W} = \sup\{ |Wv| : v\in \RR^d, |v|=1 \}. \]
If $\{e_i\}_{i=1}^d$ is the standard orthonormal basis of $\RR^d$,
then
\begin{equation} \label{eqn:alt-op-norm}
 \op{W} \approx \sum_{i=1}^d |We_i|; 
\end{equation}
the implicit constants depend only on $d$.
Let $\Ss_d$
denote the collection of all $d\times d$ symmetric, positive
definite matrices.  Though matrices in general do not commute,
matrices in $\Ss_d$ commute in operator norm:  if $W,\,V\in \Ss_d$,
then $\op{WV}= \op{VW}$. 

By a matrix weight we mean a function $W : \rn
\rightarrow \Ss_d$ whose entries are measurable scalar functions.
    For $1 < p < \infty$, a matrix weight $W$ is in $\A_p$,
    denoted by $W \in \mathcal{A}_p$, if
    \[
      [W]_{\mathcal{A}_p}
      :=
      \avgint_Q \left(\avgint_Q \op{W^{\frac{1}{p}}(x)
          W^{-\frac{1}{p}}(y)}^{p'}\,dy\right)^{\frac{p}{p'}}\,dx < \infty,
    \]
    where the supremum is taken over cubes $Q$. When $p = 1$, we say that $W \in \A_1$ if 
    \[
        [W]_{\mathcal{A}_1} := \sup_Q\esssup_{x \in Q}\avgint_Q{\op{W(y)W^{-1}(x)}\,dy} < \infty.
    \]

    A proof of the following result can be found in~\cite{MR2015733} (see
also~\cite{mb-dcu-preprint}).

\begin{prop} \label{prop:reducing-defn}
Fix $1\leq p<\infty$.  Given a matrix weight $W$,  we can define a norm $\rho_W$
on $\RR^d$ as follows: given a cube $Q$,  for any vector $v\in \RR^d$ let
\[ \rho_{W,p}(v) = \bigg(\avgint_Q |W^{\frac{1}{p}}(x)v|^p\,dx\bigg)^{\frac{1}{p}}. \]
Then there exists a constant matrix $\W_Q^p$, called the reducing
matrix of $\rho_{W,p}$, such that for all
$v$,
\[ \rho_{W,p}(v) \approx |\W_Q^p v|. \]
Moreover, if $W\in \mathcal{A}_p$ and  if we let $ \overline{\W}_Q^{p'}$ denote the reducing matrix of
the norm $\rho_{W^{-p'/p},p'}$, then we have that
\[ \sup_Q \op{\W_Q^p \overline{\W}_Q^{p'}} \approx [W]_{\A_p}^{\frac{1}{p}}, \]
and the implicit constants depend only on $d$ and $p$.
\end{prop}

There is a connection between matrix $\A_p$ weights and scalar $A_p$
weights.  The following result is proved in~\cite{MR1928089}; see
also~\cite{mb-dcu-preprint}.  

\begin{prop} \label{eqn:matrix-scalar-Ap}
Given $1\leq p <\infty$ and a matrix weight $W\in A_p$, for every
vector $v\in \RR^d$, the scalar weight $w_v(x) =|W^{\frac{1}{p}}(x)v|^p$
is in $A_p$, and $[w_v]_{A_p}\leq [W]_{\mathcal{A}_p}$.  
\end{prop}

\begin{remark}
In light of Proposition~\ref{eqn:matrix-scalar-Ap}, given $W\in \mathcal{A}_p$,
we define its scalar $A_\infty$ characteristic by 
\[ [W]_{A_\infty^{sc}} := \sup\{ [w_v]_{A_\infty} : v \in \RR^d \}.  \]
\end{remark}

\medskip

We also define fractional matrix weights.  Given $0<\alpha<n$ and $1\leq p<\frac{n}{\alpha}$,
define $q$ by $\frac{1}{p}-\frac{1}{q}=\frac{\alpha}{n}$.  A matrix
weight $W$ is in  $\mathcal{A}_{p,q}$, $1<p<\infty$, if 
\[
  [W]_{\mathcal{A}_{p,q}} :=
  \sup_Q \avgint_Q \left(\avgint_Q \op{W(x)W^{-1}(y)}^{p'}\,dy\right)^{\frac{q}{p'}}\,dx < \infty.
\]
When $p=1$, $W\in \A_{1,q}$ if
\[
  [W]_{\mathcal{A}_{1,q}} :=
  \sup_Q \esssup_{x\in Q}
 \avgint_Q \op{W(y)W^{-1}(x)}^{q}\,dy < \infty.
\]
These weights were introduced in~\cite{MR4030471} with a somewhat
different defintion.  Our definition is equivalent to theirs if we
replace the weight $W$ by $W^{\frac{1}{q}}$.

We also have a reducing operator characterization of $\A_{p,q}$.
To avoid confusion, since our definition of $A_p$ and $A_{p,q}$ are
different, we introduce different notation.  For a proof,
see~\cite[Section~2]{MR4030471}. 

\begin{prop} \label{prop:frac-reducing}
   Given $0<\alpha<n$ and $1\leq p<\frac{n}{\alpha}$,
define $q$ by $\frac{1}{p}-\frac{1}{q}=\frac{\alpha}{n}$.   Let $W\in
\A_{p,q}$.  Then for
every cube $Q$ there exist constant matrices $\V_Q^p$ and
$\overline{\V}_Q^{p'}$ such that for every $v\in \RR^d$,
\[ |\V_Q^pv|\approx \bigg(\avgint_Q |W(x)v|^q\,dx\bigg)^{\frac{1}{q}},
  \quad
  |\overline{\V}_Q^{p'}| \approx
  \bigg(\avgint_Q |W^{-1}(x)v|^{p'}\,dx\bigg)^{\frac{1}{p'}}. \]
Moreover,
\[ \sup_Q \op{\V_Q^p \overline{\V}_Q^{p'}}
  \approx [W]_{\mathcal{A}_{p,q}}^{\frac{1}{q}}. \]
In each case, the implicit constants depend only $d$, $p$, and
$\alpha$.  
\end{prop}

There is a connection between matrix and scalar $A_{p,q}$ weights.
The following result was proved in~\cite[Corollary~3.3]{MR4030471}.

\begin{prop} \label{prop:norm-scalar-Apq}
  Given $0<\alpha<n$ and $1\leq p<\frac{n}{\alpha}$,
define $q$ by $\frac{1}{p}-\frac{1}{q}=\frac{\alpha}{n}$.   If $W\in
\A_{p,q}$, then for every vector $v\in \RR^d$, the scalar weight
$w_v(x)=|W(x)v|^q$ is in $A_{p,q}$, and
\[ [W^q]_{A_\infty^{sc}} := \{ [w_v]_{A_\infty} : v \in \RR^d \} \leq
[W]_{\mathcal{A}_{p,q}}. \]
\end{prop}

\medskip

We will consider the following weighted maximal operators which were
introduced in~\cite{MR2015733,MR4030471}.  If $1\leq
p<\infty$ and  $W\in \A_p$, given a vector-valued function $f \in
L^1_{loc}(\rn,\RR^d)$, define the Christ-Goldberg maximal
operator by
\[ M_Wf(x) := \sup_Q \avgint_Q
  |W^{\frac{1}{p}}(x)W^{-\frac{1}{p}}(y)f(y)|\,dy \cdot \rchi_Q(x).  \]
where the supremum is taken over all cubes $Q$. Similarly, given $0<\alpha<n$, $1<p<\frac{n}{\alpha}$, and $W\in
\A_{p,q}$, define the fractional Christ-Goldberg maximal operator by
\[ M_{W,\alpha} f(x) := \sup_Q |Q|^{\frac{\alpha}{n}}\avgint_Q
  |W(x)W^{-1}(y)f(y)|\,dy \cdot \rchi_Q(x).  \]
%


  

\section{Proof of Theorems~\ref{thm:sio-scalar}
  and~\ref{thm:fractional-scalar}}
\label{section:scalar}

In this section we prove our two scalar results, Theorems~\ref{thm:sio-scalar}
  and~\ref{thm:fractional-scalar}.  For both we will make use of the
  theory of sparse domination.   For complete details, we refer the
  reader
  to~\cite{CruzUribe:2016ji,MR3521084,MR4007575,lerner-IMRN2012,MR3127380}.
  Hereafter, let $\D$ denote a dyadic grid, and $\Ss\subset \D$ a sparse subset:
  that is, for every $Q\in \Ss$, there exists a set $E_Q\subset Q$,
  such that the sets $E_Q$ are pairwise disjoint and $ |Q| \leq 2
  |E_Q|$.
  
\subsection*{Proof of Theorem~\ref{thm:sio-scalar}}
Given a Calder\'on-Zygmund singular integral $T$ and $f\in L_c^\infty$, there
  exist a collection of dyadic grids $
  \{\D_k\}_{k=1}^{3^n}$ and sparse families $\Ss_k \subset \D_k$ such that
  \[ |Tf(x)| \leq C(T,n)\sum_{k=1}^{3^n} A_{\Ss_k}(|f|)(x), \]
  where $A_\Ss$ is the so-called sparse operator
\[
    A_{\mathcal{S}}f(x) = \sum_{Q \in \mathcal{S}}\avg{f}_Q\cdot \rchi_Q(x).
\]
The same estimate holds with the singular integral $T$ replaced by the
Hardy-Littlewood maximal operator $M$.  Therefore, to prove
Theorem~\ref{thm:sio-scalar} we will prove the following result.

\begin{theorem} \label{thm:sparse}
  Fix a dyadic grid $\D$ and a sparse set $\Ss\subset \D$.   Given $p \geq 1$ and $w \in A_p$, then for every $f \in L^p$ and $\lambda > 0$,
    \[
        |\{x \in \mathbb{R}^d : w(x)^{\frac{1}{p}}A_{\mathcal{S}}(\cdot\,w^{-\frac{1}{p}})(x) > \lambda \} \leq [w]_{A_p}[w]_{A_\infty}^p\left(\frac{\|f\|_{L^p}}{\lambda}\right)^p.
    \]
\end{theorem}
\begin{proof}
  To prove this result we will use the equivalence
    \[
        \|w^{\frac1p}A_\mathcal{S}(w^{-\frac1p}\,\cdot)\|_{L^p \rightarrow L^{p,\infty}} \approx \sup_{\|f\|_{L^p} = 1}\,\sup_{0<|E|<\infty}\,\inf_{\substack{E'\subseteq E \\ |E| \leq 2 |E'|}}|E|^{\frac{1}{p}-1}|\langle w^{\frac1p}A_\mathcal{S}(w^{-\frac1p}f),\rchi_{E'}\rangle| 
    \]
    (see~\cite[Exercise~1.4.14]{grafakos08a}). Fix a function
    $f \in L^p$ with $\|f\|_{L^p} = 1$; since $A_\Ss$ is a positive
    operator, without loss of generality we may assume that $f$ is
    non-negative.  Let $E \subset \mathbb{R}^n$ with
    $0 < |E| < \infty$. For a positive constant $K$ (to be fixed
   below), let
    \[
        \Omega := \{x \in \mathbb{R}^n : M^\D(f^p)(x) > K/|E|\}.
      \]
where $M^\D$ is the dyadic maximal operator defined with respect to the
dyadic grid $\D$.  Form the Calderón-Zygmund decomposition of $f^p$
(with respect to the same grid $\D$) at height $K/|E|$; then we obtain a
collection of disjoint cubes $\{Q_j\}$ in $\D$ and functions $g$ and $b$ such that 
    \[
        \Omega = \bigcup_j Q_j;\quad f^p = g + b;\quad \|g\|_{L^1} \lesssim 1;\quad \|g\|_{L^\infty} \lesssim K/|E|;\quad \text{supp}(b) \subset \Omega; \quad \avg{b}_{Q_j} = 0.
    \]
    Since $M^\D$ is weak-type $(1,1)$ with
    $\|M^\D\|_{L^1\rightarrow L^{1,\infty}}=1$, and $\|f\|_{L^p} = 1$,
    if we fix $K>2$, 
    \[
      |\Omega| = |\{x \in \mathbb{R}^n : M^D(f^p)(x) > K/|E|\}
      \leq \|M^\D\|_{L^1\rightarrow L^{1,\infty}}|E|/K < |E|/2.
    \]
    Let $E' := E \setminus \Omega$; then $|E'| > |E|/2$.

    Since $w\in A_p$, by Proposition~\ref{prop:sharp-rhi}, let
    $\nu = 1 + c[w]_{A_\infty}^{-1}$ be the sharp reverse Hölder
    exponent of $w$. Fix $r$ such that $r' = p\nu' + 1$. One can
    easily verify that $r$ satisfies
    \[
        1 < r < \nu;\qquad (r')^r \lesssim \nu' \lesssim [w]_{A_\infty};\qquad \frac{(pr)'}{(p\nu)'} =  \frac{1}{p} + \frac{1}{(p\nu)'} = r.
    \]

    We can now apply  Hölder's inequality twice with exponents $p$ and $p\nu$, the definition of$A_p$, 
    the reverse H\"older inequality,  and the sparseness of the
    collection $\mathcal{S}$ to show that
    \begin{align}
      |\langle
      w^{\frac1p}A_\mathcal{S}(fw^{-\frac1p}),\rchi_{E'}\rangle|
      &=  \sum_{Q \in \mathbb{S}}\avg{fw^{-\frac{1}{p}}}_Q
        \avg{w^\frac{1}{p}\rchi_{E'}}_Q|Q| \notag\\
      &= \sum_{Q \in
        \mathcal{S}}\avg{f}_{p,Q}\avg{w^{-\frac{1}{p}}}_{p',Q}
        \avg{w}_{\nu,Q}^\frac{1}{p}\avg{\rchi_{E'}}_{(p\nu)',Q}|Q|\notag
      \\
      &\leq [w]_{A_p}^\frac{1}{p} \sum_{Q \in \mathcal{S}}
        \avg{f}_{p,Q}\avg{w}_Q^{-\frac{1}{p}}\avg{w}_{\nu,Q}^\frac{1}{p}
        \avg{\rchi_{E'}}_{(p\nu)',Q}|Q| \notag\\
      &\lesssim [w]_{A_p}^\frac{1}{p} \sum_{Q \in
        \mathcal{S}}\avg{f}_{p,Q}
        \avg{\rchi_{E'}}_{(p\nu)',Q}|E_Q|. \label{eqn:key-step}
    \end{align}
    
    Now let $Q \in \mathcal{S}$. If $Q \subset \Omega$, then
    $\avg{\rchi_{E'}}_Q = 0$, since $E' \cap \Omega = \emptyset$;
    therefore, those non-zero terms in the sum above
    correspond to  $Q$ that  intersect $\rn \setminus \Omega$. For such $Q$, if
    $Q \cap Q_j \neq \emptyset$ then either $Q \subseteq Q_j$ or
    $Q_j \subseteq Q$; since $Q_j \subset \Omega$, we must have that
    $Q_j \subset Q$.  Hence, 
    \begin{equation}\label{eqn:decomp-est}
      \avg{f}_{p,Q} = (\avg{g}_Q + \avg{b}_Q)^\frac{1}{p}
      = (\avg{g}_Q + |Q|^{-1}\sum_{Q_j \subseteq Q}
      b(Q_j))^\frac{1}{p}
      = \avg{g}_Q^\frac{1}{p},
    \end{equation}
    since $\avg{b}_{Q_j} = 0$ for any $Q_j$.  Therefore, we can
    estimate the final term above as follows: by H\"older's
    inequality with exponent $pr$, and the norm bound for the
    maximal operator,

    \begin{align*}
    \sum_{Q \in
        \mathcal{S}}\avg{f}_{p,Q}
      \avg{\rchi_{E'}}_{(p\nu)',Q}|E_Q|
      & =  \sum_{Q \in \mathcal{S}}\avg{g}_Q^\frac{1}{p}\avg{\rchi_{E'}}_{(p\nu)',Q}|E_Q|\\
        &\leq \sum_{Q \in \mathcal{S}}\int_{E_Q}(Mg)^\frac{1}{p}(M\rchi_{E'})^\frac{1}{(p\nu)'}\,dx\\
        &\leq \|Mg\|_{L^r}^\frac{1}{p}\|M\rchi_{E'}\|_{L^{r}}^\frac{1}{(p\nu)'}\\
        &\lesssim (r')^\frac{1}{p}(r')^\frac{1}{(p\nu)'}\|g\|_{L^r}^\frac{1}{p}|E'|^\frac{1}{(pr)'}\\
        &\lesssim[w]_{A_\infty}(\|g\|_{L^\infty}^\frac{1}{r'}\|g\|_{L^1}^\frac{1}{r})^\frac{1}{p}|E'|^\frac{1}{(pr)'}\\
        &= [w]_{A_\infty}|E|^{\frac{1}{(pr)'}-\frac{1}{pr'}}\\
        &\lesssim [w]_{A_\infty}|E|^{1-\frac{1}{p}}.
    \end{align*}
    If we combine these two estimates, we get
    \[
        \|w^{\frac1p}A_\mathcal{S}(\cdot\,w^{-\frac1p})\|_{L^p \rightarrow L^{p,\infty}} \lesssim [w]_{A_p}^\frac{1}{p}[w]_{A_\infty},
    \]
    as desired.
  \end{proof}

  \subsection*{Proof of Theorem~\ref{thm:fractional-scalar}}
  As with singular integrals and the Hardy-Littlewood maximal
  operator, we have the following pointwise domination estimate: given
  $0<\alpha<n$ and non-negative $f\in L^\infty_c$, there exist a
  collection of dyadic grids $ \{\D_k\}_{k=1}^{3^n}$ and sparse families
  $\Ss_k \subset \D_k$ such that
  \[ |I_\alpha f(x)| \leq C(\alpha,n)\sum_{k=1}^{3^n}
    A_{\mathcal{S}_k}^\alpha(|f|)(x), \]
  where $A_\Ss^\alpha$ is the fractional  sparse operator
\[
    A_{\mathcal{S}}f(x) = \sum_{Q \in \mathcal{S}}|Q|^{\frac{\alpha}{n}}\avg{f}_Q\cdot \rchi_Q(x).
\]
The same estimate holds with $I_\alpha$ replaced by the
fractional maximal operator $M_\alpha$.  Therefore, to prove
Theorem~\ref{thm:fractional-scalar} we will prove the following result.

\begin{theorem} \label{thm:sparse-fractional}
  Fix a dyadic grid $\D$ and a sparse set $\Ss\subset \D$.   Fix
  $0<\alpha<n$ and  $1 \leq p < \frac{n}{\alpha}$; define $q$ by
  $\frac{1}{p}-\frac{1}{q}=\frac{\alpha}{n}$.  Let $w \in A_{p,q}$;
  then  for every non-negative $f \in L^p$ and $\lambda > 0$,
    \[
        |\{x \in \mathbb{R}^d : w(x)A_{\mathcal{S}}^\alpha(fw^{-1})(x) > \lambda \} \leq [w]_{A_{p,q}}[w^q]_{A_\infty}^q\left(\frac{\|f\|_{L^p}}{\lambda}\right)^q.
    \]
  \end{theorem}

 Let $M_\alpha^\D$ denote the version of the fractional maximal
operator defined with respect to cubes in the dyadic grid $\D$.  To
prove Theorem~\ref{thm:sparse-fractional}, we will need the following
estimate which was proved in~\cite[Theorem~2.3]{MR3000426}.

\begin{lemma} \label{lemma:Malpha-est}
  Let $\D$ be a dyadic grid.  Given $0<\alpha<n$ and $1<p<\frac{n}{\alpha}$ define $q$ by
  $\frac{1}{p}-\frac{1}{q}=\frac{\alpha}{n}$.  Then for every $f \in
  L^p$,
  \[ \|M_\alpha^\D f\|_{L^q}
    \leq \left(1+\frac{p'}{q}\right)^{1-\frac{\alpha}{n}} \|f\|_{L^p}. \]
\end{lemma}

\begin{proof}[Proof of Theorem~\ref{thm:sparse-fractional}]
    We first prove the case for $p > 1$. To prove this theorem, we will again use the equivalence
    \[
        \|wA^\alpha_\mathcal{S}(\cdot\,w^{-1})\|_{L^p \rightarrow L^{q,\infty}} \approx \sup_{\|f\|_{L^p} = 1}\,\sup_{0<|E|<\infty}\,\inf_{\substack{E'\subseteq E \\ |E| \leq 2 |E'|}}|E|^{\frac{1}{q}-1}|\langle wA^\alpha_\mathcal{S}(fw^{-1}),\rchi_{E'}\rangle|.
    \]
    We now argue as we did in the proof of Theorem~\ref{thm:sparse}.
    Fix $f \in L^p$, $\|f\|_{L^p} = 1$, and
    $E \subset \mathbb{R}^n$ with $0 <|E| < \infty$.  Define the
    sets $\Omega$ and $E'$ and functions $g$ and $b$ exactly as we
    did before. Since $w\in A_{p,q}$, $w^q \in A_\infty$ so we can again
    apply Proposition~\ref{prop:sharp-rhi} and let $\nu = 1 + c[w^q]_{A_\infty}^{-1}$ be the sharp reverse
    Hölder exponent of $w^q$.  Fix $r$  such that $r' = q\nu' +
    1$; then  we again have that $r$ satisfies
    \[
        1 < r < \nu;\qquad (r')^r \lesssim \nu' \lesssim [w^q]_{A_\infty};\qquad  \frac{(qr)'}{(q\nu)'} = \frac{1}{q} + \frac{1}{(q\nu)'} = r. 
    \]
      
    Let $\beta := \alpha p$ and fix $t$  such that
    $\frac{1}{t} - \frac{p}{qr} = \frac{\beta}{n}$.  Then
    $0 < \beta < n$ and $1 \leq t < n/\beta$. Hence, by Lemma~\ref{lemma:Malpha-est} the fractional
    maximal operator $M_\beta^\D$ is bounded from $L^t$ to $L^{qr/p}$
    with operator norm at most
    \[ \left(1+ \frac{t'p}{qr}\right)^{1-\frac{\beta}{n}}
      =
      \bigg(t'\bigg(1-\frac{\beta}{n}\bigg)\bigg)^{1-\frac{\beta}{n}}.
    \]
    By the definition of $\beta$,  $1 - \frac{p}{q} =
    \frac{\beta}{n}$, so by the definition of $t$ we have that
    $\frac{p}{qr'} = \frac{1}{t'}$.  Thus
    \[
       \bigg(t'\bigg(1-\frac{\beta}{n}\bigg)\bigg)^{1-\frac{\beta}{n}}
       \leq (r')^{\frac{p}{q}}.
      \]
      
    We can now argue as before, using Hölder's inequality, the
    definition of $A_{p, q}$, the 
    reverse H\"older inequality, equation~\eqref{eqn:decomp-est},
    H\"older's inequality with exponent $qr$, and the
    norm bounds for the maximal and fractional maximal operators
    to get
    \begin{align*}
        |\langle wA^\alpha_\mathcal{S}(fw^{-1}),\rchi_{E'}\rangle| &= \sum_{Q \in \mathcal{S}}\avg{fw^{-1}}_Q\avg{w\rchi_{E'}}_Q|Q|^{1+\frac{\alpha}{n}}\\
        &= \sum_{Q \in \mathcal{S}}\avg{f}_{p,Q}\avg{w^{-1}}_{p',Q}\avg{w}_{(q\nu),Q}\avg{\rchi_{E'}}_{(q\nu)',Q}|Q|^{1+\frac{\alpha}{n}}\\
        &\leq [w]_{A_{p,q}}^\frac{1}{q}\sum_{Q \in \mathcal{S}}\avg{f}_{p,Q}\avg{w}_{q,Q}^{-1}\avg{w}_{(q\nu),Q}\avg{\rchi_{E'}}_{(q\nu)',Q}|Q|^{1 + \frac{\alpha}{n}}\\
        &\lesssim [w]_{A_{p,q}}^\frac{1}{q}\sum_{Q \in
          \mathcal{S}}\int_{E_Q}(M_\beta ^\D g)^\frac{1}{p}(M \rchi_{E'})^\frac{1}{(q\nu)'}\,dx\\
        &\leq [w]_{A_{p,q}}^\frac{1}{q}\|M_\beta ^\D g\|_{L^{qr/p}}^\frac{1}{p}\|M\rchi_{E'}\|_{L^r}^\frac{1}{(q\nu)'}\\
        &\leq [w]_{A_{p,q}}^\frac{1}{q}(r')^\frac{1}{q}(r')^\frac{1}{(q\nu)'}\|g\|_{L^t}^\frac{1}{p}|E'|^\frac{1}{(qr)'}\\
        &\lesssim [w]_{A_{p,q}}^\frac{1}{q}[w^q]_{A_\infty}(\|g\|_{L^\infty}^\frac{1}{t'}\|g\|_{L^1}^\frac{1}{t})^\frac{1}{p}|E'|^\frac{1}{(qr)'}\\
        &\lesssim [w]_{A_{p,q}}^\frac{1}{q}[w^q]_{A_\infty}|E|^{\frac{1}{(qr)'}-\frac{1}{pt'}}\\
        &= [w]_{A_{p,q}}^\frac{1}{q}[w^q]_{A_\infty}|E|^{1 - \frac{1}{q}}.
    \end{align*}
    Therefore, 
    \[
        \|w A^\alpha_\mathcal{S}(\cdot\,w^{-1})\|_{L^p \rightarrow L^{q,\infty}} \lesssim [w]_{A_{p,q}}^\frac{1}{q}[w^q]_{A_\infty},
    \]
    as desired.

    \medskip

    The proof of  this result when $p=1$, is essentially the
    same.   But, instead of using H\"older's
    inequality to replace  $\langle fw^{-1}\rangle_{Q}$ with
    $\langle f\rangle_{p,Q}\langle w^{-1}\rangle_{p',Q}$, use the estimate
    $\langle fw^{-1}\rangle_{Q}\leq \|w^{-1}\|_{L^\infty(Q)} \langle
      f\rangle_{Q}$.  The argument then continues as before, using the
      definition of $A_{1,q}$.  
\end{proof}

\section{Proof of Theorems~\ref{thm:sio-matrix}
  and~\ref{thm:fractional-matrix}}
\label{section:matrix}

In this section we prove our two matrix weighted results, Theorems~\ref{thm:sio-matrix}
  and~\ref{thm:fractional-matrix}.  We will again make use of the
  theory of sparse domination and again in this section $\D$ will
  denote a dyadic grid and $\Ss$ a sparse subset of $\D$.

  \subsection*{Proof of Theorem~~\ref{thm:sio-matrix}}
  We will first prove this for singular integrals.   The proof for the
  Christ-Goldberg maximal operator is nearly identical and is given below.
  We adapt the argument used in~\cite{MR4269407} when $p=1$.  Fix
  $1<p<\infty$ and $W\in \A_p$.  As a
  consequence of the convex body sparse domination theorem of Nazarov,
  Petermichl, Treil, and Volberg~\cite{MR3689742}, if $T$ is a
  Calder\'on-Zygmund singular integral,  given a function
  $f\in L^\infty_c(\rn,\RR^d)$, there exist sparse
  families $\Ss_j$, $1\leq j\leq 3^d$, such that
  \begin{align*}
    |W^{\frac{1}{p}}(x)T(W^{-\frac{1}{p}}f)(x)|
    & =
        \bigg|c \sum_{j=1}^{3^d} \sum_{Q\in \Ss_j}
        \avgint_Q k_Q(x,y)W^{\frac{1}{p}}(x)W^{-\frac{1}{p}}(y)f(y)\,dy
      \cdot \rchi_Q(x)\bigg| \\
    & \leq c \sum_{j=1}^{3^d} \sum_{Q\in \Ss_j} 
      \op{W(x)^{\frac{1}{p}}(\W_Q^p)^{-1}}\avgint_Q
      |\W_Q^pW^{-\frac{1}{p}}(y)f(y)|\,dy
      \cdot \rchi_Q(x).
  \end{align*}
Here, each $k_Q(x,y)$ is a scalar function such that $\|k_Q\|_\infty\leq
1$, and $\W_Q^p$ is the reducing matrix from
Proposition~\ref{prop:reducing-defn} associated to the matrix
$W$ and exponent $p$.  

To estimate the integral on the right-hand side, we apply H\"older's
inequality, the fact that matrices in $\Ss_d$ commute in norm,  and
Proposition~\ref{prop:reducing-defn} to get
\begin{align} \label{eqn:reducing-est}
       \avg{|\W_Q^pW^{-\frac{1}{p}}f|}_Q
       &\leq \avg{\op{\W_Q^pW^{-\frac{1}{p}}}|f|}_Q\\
         &\leq \avg{|f|}_{p,Q}\left(\avgint_Q
           \op{\W_Q^pW^{-\frac{1}{p}}(x)}^{p'}\,dx\right)^{\frac{1}{p'}}
  \notag \\
         &= \avg{|f|}_{p,Q}\left(\avgint_Q
           \op{W^{-\frac{1}{p}}(x)\W_Q^p}^{p'}\,dx\right)^{\frac{1}{p'}}
  \notag \\
  &\lesssim \avg{|f|}_{p,Q} \sum_{i=1}^d \left(\avgint_Q
    |W^{-\frac{1}{p}}(x)\W_Q^pe_i|^{p'}\,dx\right)^{\frac{1}{p'}}
    \notag \\
  &\lesssim \avg{|f|}_{p,Q} \sum_{i=1}^d
    |\overline{\W}_Q^{p'}\W_Q^pe_i|  \notag \\
       &\lesssim \avg{|f|}_{p,Q}\op{\overline{\W}_Q^{p'}\W_Q^p}  \notag
  \\
  & \lesssim \avg{|f|}_{p,Q} [W]_{\A_p}^{\frac{1}{p}}.  \notag
\end{align}
    
If we combine this estimate with the previous one, we see that 
    \[
      |W^{\frac{1}{p}}(x)T(W^{-\frac{1}{p}}f)(x)|
      \lesssim [W]_{\mathcal{A}_p}^\frac{1}{p}\sum_{Q \in \mathcal{S}}
      \op{W^\frac{1}{p}(x)(\W_Q^p)^{-1}}\avg{|f|}_{p,Q}\rchi_Q(x).
    \]
    Therefore, to prove  Theorem~\ref{thm:sio-matrix} it will suffice
    to prove that 
    \[
        \|\mathcal{A}_\mathcal{S}\|_{L^{p} \rightarrow L^{p, \infty}} \lesssim [W]_{A_\infty^\text{sc}},    
    \]
    for the non-negative \textit{scalar} operator $\mathcal{A}_\mathcal{S}$ defined by 
    \[
        \mathcal{A}_\mathcal{S}f(x) := \sum_{Q \in \mathcal{S}}\op{W^\frac{1}{p}(x)(\W_Q^p)^{-1}}\avg{f}_{p,Q}\rchi_Q(x).
      \]
      To prove this, we will show that we can essentially reduce the
      problem to the proof of Theorem~\ref{thm:sio-scalar}.

      Since $W\in \A_p$, by Proposition~\ref{eqn:matrix-scalar-Ap},
      for any vector $v\in \RR^d$,
      $|W ^{\frac{1}{p}} v|^p \in A_\infty$ with constant uniformly bounded by
      $[W]_{A_\infty^{sc}}$.  Therefore, by
      Proposition~\ref{prop:sharp-rhi}, if we let $\nu = 1 +
      \frac{1}{c[W]_{A_\infty}^{sc}}$,  then $|W ^{\frac{1}{p}} v|^p \in RH_\nu$.
      Therefore, we can argue as in~\cite[p.~1524]{MR4269407}:
 \begin{multline} \label{eqn:sharp-rhi-bound}
   \qquad \avgint_Q\op{W^{\frac{1}{p}}(x)(\W_Q^p)^{-1}}^{p\nu}\,dx
   \approx \sum_{i=1}^d \avgint_Q |W^{\frac{1}{p}}(x)(\W_Q^p)^{-1}e_i|^{p\nu}\,dx \\
\lesssim \sum_{i=1}^d \bigg(\avgint_Q
   |W^{\frac{1}{p}}(x)(\W_Q^p)^{-1}e_i|^p\,dx \bigg)^\nu 
     \lesssim \sum_{i=1}^d|\W_Q^p(\W_Q^p)^{-1}e_i|^{p\nu}
    = d.
  \end{multline}
 Now let $E$ and $E'$ be as in the proof of Theorem~\ref{thm:sio-scalar}.  Then we have that 
 \begin{align} \label{eqn:reduce-to-scalar}
      |\langle \mathcal{A}_\mathcal{S}f,\rchi_{E'}\rangle|
   &= \sum_{Q \in \mathcal{S}}\avg{f}_{p,Q}
     \avg{\op{W^\frac{1}{p}(\W_Q^p)^{-1}}\rchi_{E'}}_Q|Q|\\
   &\leq \sum_{Q \in \mathcal{S}}
     \avg{f}_{p,Q}\avg{\rchi_{E'}}_{(p\nu)',Q}
     \avg{\op{W^\frac{1}{p}(\W_Q^p)^{-1}}}_{p\nu,Q}|Q| \notag \\  
   &\lesssim \sum_{Q \in \mathcal{S}}\avg{f}_{p,Q}
     \avg{\rchi_{E'}}_{(p\nu)',Q}|Q|. \notag
    \end{align}
    We can now argue as we did, beginning at
    inequality~\eqref{eqn:key-step}, and with the same notation as
    before, to get 
    \[
        |\langle \mathcal{A}_\mathcal{S}f,\rchi_{E'}\rangle| \lesssim (r')^r|E|^{1 - \frac{1}{p}} \lesssim [W]_{A_\infty^\text{sc}}|E|^{1 - \frac{1}{p}}.
    \]
Therefore, again arguing as before, we conclude that 
    \[
      \|\mathcal{A}_\mathcal{S}\|_{L^p \rightarrow L^{p,\infty}}
      \lesssim [W]_{\mathcal{A}_p}^\frac{1}{p}[W]_{A_\infty^\text{sc}}.
    \]

    \medskip

The proof for the Christ-Goldberg maximal operator is nearly
identical.  By~\cite[Lemma~3.1]{MR4269407} we have that the
Christ-Goldberg maximal operator satisfies an estimate nearly
identical to that for singular integrals:  given a function
  $f\in L^1_c(\rn,\RR^d)$, there exist sparse
  families $\Ss_j$, $1\leq j\leq 3^d$, such that
  \begin{equation} \label{eqn:CG-sparse}
   M_Wf(x)
 \leq c \sum_{j=1}^{3^n} \sum_{Q\in \Ss_j} 
      \op{W(x)^{\frac{1}{p}}(\W_Q^p)^{-1}}\avgint_Q
      |\W_Q^pW^{-\frac{1}{p}}(y)f(y)|\,dy
      \cdot \rchi_Q(x).
    \end{equation}
(We will sketch a different proof of inequality \eqref{eqn:CG-sparse} for the
fractional Christ-Goldberg maximal operator below.) 
    Given this inequality, the proof proceeds exactly as before.
    
\subsection*{Proof of Theorem~\ref{thm:fractional-matrix}}
The proof of Theorem~\ref{thm:fractional-matrix} is very similar to
the proof of Theorem~\ref{thm:sio-matrix}; again the proof involves
reducing to a scalar sparse operator.  We begin by proving a sparse
domination result for the fractional integral operator with matrix
weights.  Parts of the proof are very similar to the proof
in~\cite{CruzUribe:2016ji} of the sparse
domination theorem for the fractional integral itself, so we will only
sketch the changes necessary.

\begin{prop} \label{prop:fractional-matrix-sparse}
  Given $0<\alpha<n$ and $1\leq p<\frac{n}{\alpha}$,
define $q$ by $\frac{1}{p}-\frac{1}{q}=\frac{\alpha}{n}$.   Let $W\in
A_{p,q}$.  Then there exist $3^n$ dyadic grids $\D_j$ such that if
$f\in L^p(\rn,\RR^d)$ has compact support, then for almost every $x\in \rn$,
\begin{equation}  \label{eqn:fms1}
|W(x)I_\alpha(W^{-1}f)(x)|
  \leq
  C(n,\alpha) \sum_{j=1}^{3^n} I_\alpha^{\D_j}(|W(x)W^{-1}f|)(x), 
\end{equation}
where
\[ I_\alpha^{\D_j}(|W(x)W^{-1}f|)(x)
  = \sum_{Q\in \D_j} |Q|^{\frac{\alpha}{n}} \avg{|W(x)W^{-1}f|}_Q\cdot
  \rchi_Q(x). \]
Further, for each such $f$ and any dyadic grid $\D$, there exists a sparse
family $\Ss \subset \D$ (depending on $f$) such that for almost every $x\in \rn$,
\begin{equation}  \label{eqn:fms2}
  I_\alpha^{\D}(|W(x)W^{-1}f|)(x)
  \leq
  [W]_{\mathcal{A}_{p,q}}^{\frac{1}{q}} \sum_{Q\in \Ss}
    \op{W(x)(\V_Q^q)^{-1}} |Q|^{\frac{\alpha}{n}}
    \avg{|f|}_{p,Q}\rchi_Q(x). 
  \end{equation}
\end{prop}

\begin{proof}
  The proof of inequality~\eqref{eqn:fms1} is essentially the same as
  the proof for unweighted scalar fractional integrals
  in~\cite[Proposition~3.3]{CruzUribe:2016ji}.   By the definition of
  the fractional integral operator, for $x\in \rn$, 
  \[ |W(x)I_\alpha(W^{-1}f)(x)|
    \leq
    \int_\rn \frac{|W(x)W^{-1}(y)f(y)|}{|x-y|^{n-\alpha}}\,dy.  \]
  We claim that for almost every $x\in \rn$, $|W(x)W^{-1}(\cdot)f(\cdot)|$ is
   integrable.  Fix any large cube $Q$ containing the support of $f$;
  then by H\"older's inequality,
  \[ \int_Q |W(x)W^{-1}(y)f(y)|\,dy 
    \leq
    \bigg(\int_Q \op{W(x)W^{-1}(y)}^{p'}\,dy\bigg)^{\frac{1}{p'}}
    \|f\|_{L^p(\rn)}.  \]
  By the definition of $A_{p,q}$, the integral on the righthand side
  is finite for almost every $x\in Q$.  Since this is true for any
  increasing sequence of such cubes, we get that $|W(x)W^{-1}(\cdot)f(\cdot)|$  is in
  $L^1$ for almost every $x$.  Given this, we can now apply the
  decomposition argument in ~\cite[Proposition~3.3]{CruzUribe:2016ji}
  and we get the desired estimate.

  We will now prove inequality~\eqref{eqn:fms2} in two steps.  First,
  we will show that given $\D$ and $f$, there exists a sparse set
  $\Ss$ such that
  \[  I_\alpha^{\D}(|W(x)W^{-1}f|)(x)
    \leq
    C(n,\alpha)\sum_{Q\in \Ss} |Q|^{\frac{\alpha}{n}} \avg{|W(x)W^{-1}f|}_Q\cdot
    \rchi_Q(x).  \]
  The proof is essentially the same as the proof of
  ~\cite[Proposition~3.6]{CruzUribe:2016ji}.  To repeat this argument,
  we need to show that for almost every $x\in \rn$,
  \[ \avgint_Q |W(x)W^{-1}(y)f(y)|\,dy \rightarrow 0 \]
  as $|Q| \rightarrow \infty$.   But, since $
  |W(x)W^{-1}(\cdot)f(\cdot)|\in L^1(\R^n)$, this follows
  immediately.

  The second step to prove  inequality~\eqref{eqn:fms2} is to repeat
  the argument used to prove~\eqref{eqn:reducing-est}.   Let $\V_Q^p$
  and $\overline{\V_Q^{p'}}$ be the reducing operators associated to
  $W$.  Then by Proposition~\ref{prop:frac-reducing}, we have that for
  each cube $Q$,
  \begin{align*}
    \avg{|W(x)W^{-1}f|}_Q
    & \leq
    \op{W(x)(\V_Q^q)^{-1}} \avgint_Q |\V_Q^q W^{-1}(y)f(y)|\,dy \\
    &  \leq \op{W(x)(\V_Q^q)^{-1} }\avg{|f|}_{p,Q}
        \bigg(\avgint_Q \op{\V_Q^q
          W^{-1}(y)}^{p'}\,dy\bigg)^{\frac{1}{p'}} \\
     &   \lesssim \op{W(x)(\V_Q^q)^{-1}} \avg{|f|}_{p,Q}
          \op{\overline{\V}_Q^{p'}\V_Q^q} \\
      &    \lesssim [W]_{\mathcal{A}_{p,q}}^{\frac{1}{q}} \op{W(x)(\V_Q^q)^{-1}}\avg{|f|}_{p,Q}. 
  \end{align*}
  If we now combine these estimates, we get~\eqref{eqn:fms2}.
\end{proof}


We can now prove Theorem~\ref{thm:fractional-matrix} for the
fractional integral operator.  The proof is very similar to the proof
of Theorem~\ref{thm:sio-matrix} and we sketch the changes.   By
Proposition~\ref{prop:fractional-matrix-sparse} it will suffice to
prove that
    \[
      \|\mathcal{A}_\mathcal{S}\|_{L^{p} \rightarrow L^{p, \infty}}
      \lesssim [W^q]_{A_\infty^\text{sc}},    
    \]
    for the non-negative scalar operator $\mathcal{A}^\alpha_\mathcal{S}$ defined by 
    \[
        \mathcal{A}^\alpha_\mathcal{S}f(x) := \sum_{Q \in \mathcal{S}}\op{W^\frac{1}{p}(x)(\V_Q^q)^{-1}}\avg{f}_{p,Q}\rchi_Q(x).
      \]

 By Proposition~\ref{prop:norm-scalar-Apq}, for any vector $v\in
 \RR^d$, $|Wv|^q \in A_{p,q}$ and so in $A_\infty$.  Therefore, by
 Proposition~\ref{prop:sharp-rhi} it satisfies a reverse H\"older
 inequality and we can argue as we did in~\eqref{eqn:sharp-rhi-bound}
 to get that for any cube $Q$,
\[  \avgint_Q \op{W(x)(\V_Q^q)}^{q\nu}\,dx \lesssim d. \]
Given this, we can repeat the argument used to
prove~\eqref{eqn:reduce-to-scalar} to show that we need to bound (in
the notation used in the proof of Theorems~\ref{thm:sio-scalar}
and~\ref{thm:sio-matrix})
\[ \sum_{Q \in
    \mathcal{S}}\avg{f}_{p,Q}\avg{\rchi_{E'}}_{(q\nu)',Q}|Q|. \]
Given this the argument now proceeds exactly as before to complete the
proof.

\medskip

The proof for the fractional Christ-Goldberg maximal operator is
essentially the same.  We can adapt the proof of the sparse bounds for
the fractional maximal operator
in~\cite[Propositions~3.2,~3.5]{CruzUribe:2016ji} using the same ideas
as in the proof of Proposition~\ref{prop:fractional-matrix-sparse} to
show that if $f\in L^p(\rn,\RR^d)$, then there exist $3^n$ dyadic
grids $\D_j$ and sparse families $\Ss_j \subset \D_j$, such that
\[ M_{W,\alpha} f(x) \lesssim [W]_{\mathcal{A}_{p,q}}^{\frac{1}{q}}
  \sum_{j=1}^{3^n} \sum_{Q\in \Ss_j}
  \op{W(x)(\V_Q^q)^{-1}} |Q|^{\frac{\alpha}{n}}
  \avg{|f|}_{p,Q} \rchi_Q(x). \]
Given this estimate, the proof now proceeds as before.

\begin{remark}
  When $\alpha=0$, this argument also yields a different proof of the
  sparse bound~\eqref{eqn:CG-sparse}. 
\end{remark}

\appendix

\section{Proof of Theorem~\ref{ex:lower-bound}}
\label{section:example}

In this section we sketch the proof of a lower bound for the constant
in Theorem~\ref{thm:p=1-scalar}.  This bound is worse than the
sharp result found by~\cite{LLOR23}, but we believe it provides some
insight into the behavior of weights in multiplier weak-type inequalities.  

\begin{theorem} \label{ex:lower-bound}
  When $n=1$, given $w\in A_1$,  for the Hilbert transform we must
  have that
  \[ \| wH(\cdot\, w^{-1}) \|_{L^1 \rightarrow
      L^{1,\infty}} \gtrsim [w]_{A_1}^{\frac{1}{2}}. \]
\end{theorem}

To prove this result we will construct an explicit weight $w\in A_1$ and
function $f$ such that in the inequality
\begin{equation}\label{eqn:lower-bound}
  \lambda|\{x\in [0,\infty) : w(x)|H(fw^{-1})(x)|>\lambda\}|
  \leq C_0 \int_\RR |f(x)|\,dx.
\end{equation}
the constant satisfies $C_0 \gtrsim [w]_{A_1}^{\frac{1}{2}}$.  
To motivate our construction, we make some observations.  When $p=1$,
a lower bound for the Hilbert transform in
the weak $(1,1)$ inequality of the form~\eqref{eqn:weak-measure} is
gotten by considering the $A_1$ weights $w_\delta(x)=|x|^{\delta-1}$.  It
is straightforward to show that $[w_\delta]_{A_1} \approx
\delta^{-1}$, and to find a function $f$ such that the lower bound on
the weak $(1,1)$ constant is comparable to $\delta^{-1}$.  (See
Buckley~\cite{buckley93}.  The actual sharp constant in the weak
$(1,1)$ inequality is comparable to $[w]_{A_1}\log(e+[w]_{A_1})$:  see~\cite{MR4150264}.)

This example, however, does not work for the multiplier weak $(1,1)$
inequality.  Since Muckenhoupt and Wheeden~\cite{MR447956} showed that
the weight $w_0=|x|^{-1}$ is a good weight for this inequality, and
the weights $w_\delta$ converge to $w_0$, we do not get any dependence
on $\delta$ when we estimate the multiplier weak $(1,1)$ inequality
from below.

Therefore, to build the desired example, we need to start with a
weight that does not work for the multiplier weak $(1,1)$ inequality
and then modify this .  Define the function $\mu$ on $\RR$ by
\[ \mu(x) = \begin{cases}
    \frac{\log(\frac{e}{|x|})}{|x|}, & 0<|x|\leq 1, \\
      1, & |x|>1. 
    \end{cases}
  \]
By~\cite[Equation (5.5)]{MR447956}, a necessary condition for the weak
$(1,1)$ inequality for the Hilbert transform to hold with weight
$\mu$ is that for every
interval $Q\subset \RR$,
\[ |Q|^{-1} \|\mu\rchi_Q\|_{L^{1,\infty}} \leq C \essinf_{x\in Q}
  \mu(x). \]
For $0<t<1$ and $Q=[0,t]$, this condition implies that for all
$\lambda>0$,
\begin{equation} \label{eqn:omega-MW}
  \frac{\lambda}{t} \big|\{ x\in [0,t] :  \mu (x) >\lambda \}|
  \leq C\mu(t). 
\end{equation}

To show that this inequality does not hold, we will first approximate
the inverse of $\mu$ close to the origin.  For $0<x<1$, let
\[ \nu(x) = \frac{\log(ex)}{x}. \]
Then 
\[ \nu(\mu(x)) = x \, \frac{\log(\frac{e}{x})+
    \log\log(\frac{e}{x})}{\log(\frac{e}{x})}, \]

and so $x \leq \nu(\mu(x)) \leq 2x$.  
Hence, for any $\lambda>0$,
\[  \frac{\lambda}{t} \big|\{ x\in [0,t] : \mu(x) > \lambda \}|
\geq
  \frac{\lambda}{t} \big|\{ x\in [0,t] : 2x <  \nu(\lambda) \}|
  =
  \frac{\log(e\lambda)}{2t}.  \]
Thus,~\eqref{eqn:omega-MW} implies that there exists $C>0$
such that for all $0<t<1$
 and all $\lambda>0$,
\[ \frac{\log(e\lambda)}{2t} \leq
  C\frac{\log(\frac{e}{t})}{t}. \]
However, for any fixed $C$ and $t$, this inequality does not hold as
$\lambda\rightarrow \infty$.  

To construct our example, fix  $0<\delta<\tfrac{1}{2}$ and define $w_{\delta}$ by
\[ w_{\delta}(x) = \begin{cases}
    \frac{\log(\frac{e}{|x|})}{|x|^{1-\delta}}, & 0<|x|\leq 1, \\
    1, & |x|> 1. 
    \end{cases}
  \]
   We claim that $w_\delta\in A_1$ and
  $[w_\delta]_{A_1} \approx \frac{1}{\delta^{2}}$.  To
  prove this, since $w_\delta$ is symmetric and decreasing on $[0,\infty)$,
  using the arguments in~\cite{MR1406495} it will suffice to check the
  $A_1$ condition on intervals of the form $[0,t]$, $t>0$.
By integration by parts, we know that for all $t>0$,
\[ \frac{1}{t}\int_0^t \frac{\log(\frac{e}{x})}{x^{1-\delta}}\,dx =
    \frac{1}{\delta}\frac{\log\big(\frac{e}{t}\big)}{t^{1-\delta}} +
    \frac{1}{\delta^2}\frac{1}{t^{1-\delta}}. \]
If $t\leq 1$, then
  \[ \frac{1}{t}\int_0^t w_\delta(x)\,dx 
    \leq  \bigg(\frac{1}{\delta}+\frac{1}{\delta^{2}}\bigg)
    \frac{\log\big(\frac{e}{t}\big)}{t^{1-\delta}}
  \leq
  \frac{2}{\delta^{2}}\frac{\log\big(\frac{e}{t}\big)}{t^{1-\delta}}
  =  \frac{2}{\delta^{2}} w_\delta(t) . \]

On the other hand, if $t>1$, then
\begin{equation*}
  \frac{1}{t}\int_0^t w_\delta(x)\,dx
   = \frac{1}{t}\int_0^1 w_\delta(x)\,dx + \frac{1}{t}\int_1^t w_\delta(x)\,dx 
 = \frac{1}{t}\bigg(\frac{1}{\delta}+\frac{1}{\delta^{2}}\bigg)+ \frac{t-1}{t}
  \leq \frac{2}{\delta^{2}}w_\delta(t).
\end{equation*}
When $t=1$ the integral is equal to
$\frac{1}{\delta}+\frac{1}{\delta^{2}}$.  
Therefore, $w_\delta\in A_1$ and  $[w_\delta]_{A_1} \approx
\frac{1}{\delta^{2}}$.

\medskip

To get a lower bound on the  constant $C_0$
in~\eqref{eqn:lower-bound},  let $f(x)=\rchi_{[1,2]}(x)$.  Then for $0<x<\frac{1}{2}$,
\[ |H(fw_\delta^{-1})(x)|
  =
  \bigg|\int_1^2 \frac{w_\delta(y)^{-1}}{x-y}\,dy \bigg|
  =
  \int_1^2 \frac{dy}{y-x} > \frac{1}{2}. \]

Moreover, on the
interval $[0,\frac{1}{2}]$ we have that
\[ \frac{w_\delta(x)}{\mu(x^{1+\delta})} = \frac{\log(\frac{e}{x})}{\log(\frac{e}{x^{1-\delta}})}
  =
  \frac{1 + \log(\frac{1}{x})}{1 + (1-\delta)\log(\frac{1}{x})}
  \geq 1.\]
Therefore, for all $\lambda>1$,
\begin{multline*}
  \lambda|\{x\in [0,\infty) : w_\delta(x)|H(fw_\delta^{-1})(x)|>\lambda\}| 
  \geq
  \lambda |\{x\in [0, \tfrac{1}{2}] : w_\delta(x) > 2\lambda \}|\\
 \geq
  \lambda |\{x\in [0, \tfrac{1}{2}] :   \mu(x^{1-\delta})> 2\lambda \}|
 = \lambda \left(\frac{1}{2}\nu(2\lambda)\right)
                                                                                             ^{\frac{1}{1-\delta}}
 \geq
    \frac{1}{4}\lambda^{1-\frac{1}{1-\delta}}\log(\lambda)^{\frac{1}{1-\delta}}. 
  \end{multline*}
By a simple calculus argument, the function $F(\lambda) =
\lambda^{1-\frac{1}{1-\delta}}\log(\lambda)^{\frac{1}{1-\delta}}$ has
a unique local maximum when $\log(\lambda)=\frac{1}{\delta}$, so
$\lambda=e^{\frac{1}{\delta}}$. At this value
of $\lambda$,
\[
  \lambda^{1-\frac{1}{1-\delta}}\log(\lambda)^{\frac{1}{1-\delta}}
  = \exp\bigg(\frac{1}{\delta}\bigg(1-\frac{1}{1-\delta}\bigg)\bigg)
  \delta^{\frac{-1}{1-\delta}}
  = \exp\left(\frac{-1}{1-\delta}\right) \delta^{\frac{-\delta}{1-\delta}} \frac{1}{\delta}
  \approx \frac{1}{\delta}. \]
If we combine these estimates we see that  in~\eqref{eqn:lower-bound}, we must have that
\[ C_0 \gtrsim \frac{1}{\delta} \approx [w_\delta]_{A_1}^{\frac{1}{2}}. \]

\bibliographystyle{plain}
\bibliography{weighted-weak-type}

\end{document}